\documentclass{article}
\usepackage{amsmath}
\usepackage{amsfonts}
\usepackage{amsthm}
\usepackage{amssymb}
\usepackage[english]{babel}
\usepackage[ansinew]{inputenc}
\usepackage[all]{xy}
\usepackage{color}

\newtheorem{defi}{Definition}[section]
\newtheorem{theorem}[defi]{Theorem}

\newtheorem{lemma}[defi]{Lemma}

\newtheorem{remark}[defi]{Remark}

\author{Marco Antei, Vikram B. Mehta}

\title{Vector Bundles over Normal Varieties Trivialized by Finite Morphisms}
\begin{document}
\maketitle

\noindent \textbf{Abstract.} Let $Y$ be a normal and projective variety over an algebraically closed field $k$ and $V$ a vector bundle over $Y$. We prove that if there exist a $k$-scheme $X$ and a finite surjective morphism  $g:X\to Y$  that trivializes $V$ then $V$ is essentially finite.

\tableofcontents
\bigskip

\section{Introduction}\label{sez:Intro} Essentially finite vector bundles over a reduced, connected and proper scheme $Y$ over a perfect field $k$ have been defined by Nori in \cite{Nor} and \cite{Nor2}. They turn out to be those vector bundles $V$ over $Y$ which are trivialized by a principal $G$ bundle $f:Z\to Y$ for a certain finite $k$-group scheme $G$ (i.e. $f^{\ast}(V)$ is trivial). In \cite{BDS} Biswas and Dos Santos prove that if $Y$ is smooth and projective and $k$ is algebraically closed then $V$ is essentially finite if and only if there exist a scheme $X$ and a finite surjective morphism  $X\to Y$ trivializing $V$. Here we prove the same property only assuming $Y$ to be normal and projective.
\\\indent \textbf{Acknowledgments} The first author would like to thank Vikram Mehta for the invitation to the Tata Institute of Fundamental Research  of Mumbai where this paper has been conceived. The second author would like to thank the ICTP, Trieste for hospitality.  We have been informed that V. Balaji and A. J. Parameswaran are also independently considering the same problem.

\section{The theorem}
Throughout the whole paper  $k$ will be an algebraically closed field and $Y$ a normal and projective variety over $k$. Let us denote by $EF(Y)$ the neutral tannakian category of essentially finite vector bundles over $Y$. The aim of this paper is to prove the following

\begin{theorem}\label{theoFinite}
Assume there exist a normal projective variety $X$ over $k$ and a finite surjective morphism  $g:X\to Y$ such that  $g^{\ast}(V)$ is trivial, then $V\in EF(Y)$.
\end{theorem}

\begin{remark}This theorem holds in both zero and positive characteristic.
\end{remark}
 First we  consider two important special cases: the case where $g:X\to Y$ is purely inseparable (i.e. the extension $K(Y)\subset K(X)$ of their function fields is purely inseparable, which only occurs when $char(k)>0$) and the case where it is separable. Then it will only remain to reduce to these two cases.
 
\begin{lemma}\label{lemmaPI}
Assume there exist a normal projective variety $X$ over $k$ and a finite, surjective, purely inseparable morphism  $g:X\to Y$ such that  $g^{\ast}(V)$ is trivial, then $V\in EF(Y)$.
\end{lemma}
\begin{proof}We are in the case $char(k)=p>0$. So let us denote by $F_X:X\to X$ and $F_Y:Y\to Y$ respectively the absolute Frobenius morphisms of $X$ and $Y$.  Since $K(Y)\subset K(X)$ is purely inseparable then there exists a positive integer $n$ such that $K(X)^{(p^n)}  \subset K(Y)$. This implies that there is a morphism $h: Y \to X $  such that $gh=F_Y^n$ (i.e. the Frobenius iterated $n$ times)  and $hg =  F_X^n$. By assumption $g^* (V)$ is trivial on $X$, thus $h^*g^*(V)=(gh)^*(V)=(F_Y^{n})^{\ast}(V)$ is trivial hence $V$ is essentially finite (cf. \cite{MS1}).
\end{proof}

\begin{lemma}\label{lemmaSEP}
Assume there exist a normal projective variety $X$ over $k$ and a finite, surjective, separable morphism  $g:X\to Y$ such that  $g^{\ast}(V)$ is trivial, then $V\in EF(Y)$.
\end{lemma}
\begin{proof}
We may assume that $K(X)$  is normal (then Galois) over $K(Y)$   with Galois group $G$ (if it is not simply consider the normal closure of the extension $K(Y)\subset K(X)$). 
 Let $W:=(g_*\mathcal{O}_X)_{max}$ be the  maximal  semistable subsheaf of  $g_* \mathcal{O}_X$ (i.e. the first term of the Harder-Narasimhan filtration of $g_* \mathcal{O}_X$) then its slope $\mu (W) = \mu_{max} (g_*\mathcal{O}_X) = 0$: indeed since there is at least the canonical morphism  $\mathcal{O}_Y \to g_* \mathcal{O}_X$ then  in particular we have
$$0=\mu(\mathcal{O}_Y)\leq \mu_{max}(g_* \mathcal{O}_X);$$ but  $g$ is seperable then  $g^*(W)$ is still semistable; now consider the isomorphism 
$$Hom_X(g^*(W),\mathcal{O}_X)\simeq Hom_Y(W,g_*\mathcal{O}_X)\neq 0$$  from which we deduce $\mu(g^*(W))\leq 0$ hence $\mu (W) \leq 0$ (recall that  $\mu(W)=\mu(g^*(W))/deg(g)$). 


The coherent sheaf $W$ is in general only torsion free over $Y$. But it is locally free if restricted to  a big open subset $Y_0\subset Y$,   i.e. $codim_Y(Y\backslash Y_0) \geq 2$. Let $W_0:=W_{|Y_0}$ denote the vector bundle over $Y_0$,  $Sym^* (W_0^*)$  the symmetric algebra of the dual of $W_0$ and consider  $X_0:=\mathbf{Spec}(Sym^*(W_0^*))$ with its canonical map $g_0:X_0\to Y_0$.

The vector bundle  $W_0$  is strongly semistable of degree $0$ over $Y_0$: let us denote by $F_{X_0}$ and $F_{Y_0}$ respectively the absolute Frobenius morphisms of $X_0$ and $Y_0$ and  assume $W$ is not strongly semistable then there exists  a subsheaf  $U$ of $F_{Y_0}^* (W)$
such that $deg (U) > 0$.   Let $X_1$ be the fiber product of  $g_0 : X_0 \to Y_0$  and $F_{Y_0}$. It is an integral scheme. We denote by $pr:X_1 \to X_0$ and $g_1:X_1 \to Y_0$ the  projections and also $h:X_0\to X_1$ the normalization map:  

$$\xymatrix{
 {X_0} \ar@/_/[ddr]_{g_0} \ar@/^/[drr]^{F_{X_0}}
   \ar[dr]|-{h}            \\
  & X_1 \ar[d]_{g_1} \ar[r]^{pr}
                 & X_0 \ar[d]^{g_0}       \\
  & Y_0 \ar[r]^{F_{Y_0}}   & Y_0.                }
$$


Now  $U\subseteq F_{Y_0}^* ({g_0}_* ( \mathcal{O}_{X_{0}}))   =  {g_1}_*  (pr^* ( \mathcal{O}_{X_0}))  =  {g_1} _* ( \mathcal{O}_{X_{1}})$.    But from $\mathcal{O}_{X_1}\hookrightarrow h_*(\mathcal{O}_{X_0})$ we obtain ${g_1}_*(\mathcal{O}_{X_1})\hookrightarrow {g_1}_*(h_*(\mathcal{O}_{X_0}))={g_0}_* (\mathcal{O}_{X_0}) $ the latter being semistable whence a contradiction.
As a consequence  we have a homomorphism of $\mathcal{O}_{Y_0}$-algebras ${g_0}_*(\mathcal{O}_{X_0})\simeq W_{0}$  (cf. also \cite{BP}, \S 6).

Since ${g_0}_{\ast}(\mathcal{O}_{X_0})$  is semistable of slope $0$ over $Y_0$ then   $X_0$ is a Galois-étale cover over $Y_0$, the Galois group of $g_0$ still being $G$. Now let us fix some notations: recall that by assumption $V$ is a vector bundle over $Y$ such that $T:=g^{\ast}(V)$ is trivial on $X$; we set $V_0:=V_{|Y_0}$ and $T_0:=g_0^{\ast}(V_0)$ so the latter is also trivial on $X_0$. Since $g_0$ is a Galois-étale cover
then  $T_0$ is a $G$-bundle on $X_0$. But  $X_0$ is a big   open set in  $X$ thus $G$ acts on $X$ and then $G$ acts also on $T$.   Since $T$ is a $G$-bundle  then we go on as follows: we have $X / G \simeq Y$   and the trivial bundle  $T$ on $X$ descends to $Y$. So by Kempf's lemma  (cf. for example \cite{DN}, Th\'eor\`eme 2.3), for all $x$ in $X$, the stabilizer $G_x$   acts trivially on the fibre  $T_x$. But $T$ is trivial and both $X$ and $X_0$ have no global sections except constants, this means that there is a map 

$$\rho:G \to GL(T_x)=GL_r$$  over   $X$, where $r := rank (T)$.   Now let us assume for a moment that $G$ acts faithfully on $X$ so that the map $\rho: G \to GL_r$   is injective.  We already know that $G$ acts freely on $X_0$. So let us take $x\in X\backslash X_0$: since $G_x$ is a subgroup of $G$ then $G_x$ has to be trivial.   This  proves  that $G$ acts freely on $X$. So $g : X \to Y$ is a Galois-étale cover.  So $V$ is in $EF(Y)$. Up to now we have assumed $\rho$ to be injective. If it is not, i.e. if $G$ does not act faithfully on $X$  then just consider $H:= im (\rho)$, then consider the contracted product $X':=X\times^G H$, i.e. $X'=X/(ker(\rho))$, which is provided with a faithfull $H$-action and clearly $Y\simeq X'/H$. Hence $H\to GL_r$ is injective, $V$ is trivial over $X'_0:=X_0\times^G H$ and we proceed as before.

\end{proof}
 
 We are now ready to prove the principal result:

\begin{proof} \emph{of Theorem} \ref{theoFinite}: if $char (k)=0$ then lemma \ref{lemmaSEP} is sufficient to conclude. So let us assume $char(k)=p>0$: if  $g$ is purely inseparable then lemma \ref{lemmaPI} is enough to conclude. Otherwise, if $g$ is arbitrary,  we argue as follows: again we may assume that $K(X)$  is normal over $K(Y)$   with Galois group $G$. It is known that $L:=K(X)^G$     is a proper purely inseparable field extension of  $K(Y)$ while $K(X)$ is separable over $L$, then Galois.  Let $Z$ be the integral closure of $Y$ in $L$, then $g:X\to Y$ factors through the maps $s: X \to Z$ and $t : Z \to Y$ (i.e. $ts=g$) where $t : Z \to Y$ is purely inseparable and $s : X \to Z$ is separable. By lemma \ref{lemmaSEP} the vector bundle $W:=t^*(V)\in EF(Z)$ because $s^{\ast}(W)$ is trivial on $X$. As we did for lemma \ref{lemmaPI}, there exists a morphism $h:Y\to Z$ such that $h^*t^*(V)=(th)^*(V)=(F_Y^{n})^{\ast}(V)$ for some integer $n$; but $h^*(W)\in EF(Y)$ thus $(F_Y^{n})^{\ast}(V)\in EF(Y)$ then there exists $m\geq n$ such that $(F_Y^{m})^{\ast}(V)$ is Galois-étale trivial (i.e. there exists a Galois-étale cover $j:Y'\to Y$ such that $j^*((F_Y^{m})^{\ast}(V))$ is trivial on $Y'$) and that is enough to conclude that $V$ is essentially finite on $Y$. 
\end{proof}


\begin{thebibliography}{99}

\bibitem{BP} {\sc V. Balaji, A. J. Parameswaran}, \emph{An Analogue of the Narasimhan-Seshadri Theorem and some Applications}, arXiv:0809.3765v2.

\bibitem{BDS} {\sc I. Biswas, J. P. P. Dos Santos}, \emph{Vector Bundles Trivilized by Proper Morphisms and the Fundamental Group Scheme},  arXiv:0912.3572v1.

\bibitem{DN} {\sc J.-M. Drezet, M. S. Narasimhan}, \emph{Groupe de Picard des vari\'et\'es de modules
de fibr\'es semi-stables sur les courbes alg\'ebriques}, Invent. math. 97, 53-94 (1989)


%
\bibitem{MS1} {\sc V.B. Mehta, S. Subramanian}, \emph{On the fundamental group scheme}, Invent. math. 148, 143-150 (2002).
%
%
\bibitem{Nor} {\sc  M.V. Nori}, \emph{On the Representations of the Fundamental Group}, Compositio Matematica, Vol. 33,
Fasc. 1, (1976). p. 29-42.

\bibitem{Nor2} {\sc  M.V. Nori}, \emph{The Fundamental Group-Scheme}, Proc. Indian Acad. Sci. (Math. Sci.), Vol. 91,
Number 2, (1982), p. 73-122.
%
%

\end{thebibliography}

\medskip

\scriptsize

\begin{flushright} Marco Antei\\ 
E-mail: \texttt{antei@math.univ-lille1.fr}\\ \texttt{marco.antei@gmail.com}\\
\end{flushright}

\begin{flushright} Vikram B. Mehta\\ 
School of Mathematics,\\ Tata Institute of Fundamental Research,\\ Homi Bhabha Road,
Bombay 400005, India\\
E-mail:  \texttt{vikram@math.tifr.res.in}\\
\end{flushright}

\end{document}